\documentclass[12pt]{amsart}
\usepackage{marktext}
\usepackage{gimac}

\makeatletter
\newcommand{\leqnomode}{\tagsleft@true\let\veqno\@@leqno}
\newcommand{\reqnomode}{\tagsleft@false\let\veqno\@@eqno}
\newcommand{\mylabel}[2]{\def\@currentlabel{#2}\label{#1}}
\makeatother

\colorlet{darkred}{red!80!black}

\newcommand{\bbP}{{\mathbb P}}
\newcommand{\bbE}{{\mathbb E}}
\newcommand{\bbS}{{\mathbb S}}
\newcommand{\calC}{{\mathcal C}}
\newcommand{\Om}{\Omega}
\newcommand{\tht}{\varphi}
\newcommand{\beq}{\begin{equation}} 
\newcommand{\eeq}{\end{equation}}
\newcommand{\lb}{\label}
\newcommand{\calB}{{\mathcal B}}
\newcommand{\diff}{D}
\newcommand{\diffe}{\diff_e}

\begin{document}
\typeout{textwidth: \the\textwidth}
\title[Convection-Induced Singularity Suppression]{Convection-Induced Singularity Suppression in the Keller-Segel and Other Non-linear PDEs}

\author[Iyer]{Gautam Iyer}
\address{%
  Department of Mathematical Sciences,
  Carnegie Mellon University,
  Pittsburgh, PA 15213, USA}
\email{gautam@math.cmu.edu}

\author[Xu]{Xiaoqian Xu}
\address{%
  Department of Mathematics,
  Duke Kunshan University,
  Kunshan, China
}
\email{xiaoqian.xu@dukekunshan.edu.cn}

\author[Zlato\v s]{Andrej Zlato\v s}
\address{%
  Department of Mathematics,
  UC San Diego,
  La Jolla, CA 92130, USA}
\email{zlatos@ucsd.edu}
\thanks{
  This work has been partially supported by
  the National Science Foundation under grants
  DMS-1652284 and DMS-1900943 to AZ, and
  DMS-1814147 to GI, 
  as well as by the Center for Nonlinear Analysis.
}

\begin{abstract}
  In this paper we study the effect of the addition of a convective term, and of the resulting increased dissipation rate, on the growth of solutions to a general class of non-linear parabolic PDEs.
  In particular, we show that blow-up in these models can always be prevented if the added drift has a small enough dissipation time.
  We also prove a general result relating the dissipation time and the effective diffusivity of  stationary cellular flows, which allows us to obtain examples of simple  incompressible flows with arbitrarily small dissipation times.

As an application, we show that blow-up in the Keller-Segel model of chemotaxis can always be prevented if the velocity field of the ambient fluid has a sufficiently small dissipation time.
We also study reaction-diffusion equations with ignition-type nonlinearities, and show that the reaction can always be quenched by the addition of a convective term with a small enough dissipation time, provided the average initial temperature is initially below the ignition threshold.
\end{abstract}

\maketitle

\section{Introduction and main results}\label{s:intro}

The question of growth and blow-up of solutions to  non-linear parabolic PDEs is of widespread interest, and arises in their applications to areas such as fluid dynamics, population dynamics, combustion, and cosmology.
Convection is often included in the relevant models, and its presence may have two diametrically opposite effects on the solution dynamics.
While, on the one hand, it does have the ability to promote formation of singular structures, in this paper we concentrate on the opposite, stabilizing effect of convection.

The ability of strong (incompressible) drifts to slow down growth of solutions to non-linear PDEs, and even prevent their blow-up, has been studied by many authors (see for instance~\cite{FannjiangKiselevEA06,BerestyckiKiselevEA10,KiselevXu16,BedrossianHe17,He18,HeTadmor19}).
Many of the results focus on the study of special convective motions with certain mixing properties, and their ability to prevent singularity formation in  specific models.
While such flows are good candidates for this purpose, there are various other (often much simpler) flows that can be used to control singularities.

Indeed, our first main result shows that addition of flows that simply enhance dissipation to a certain degree is sufficient to keep solutions regular.
This, of course, includes many strong mixing flows, whose dissipation-enhancing properties have recently been studied (see, e.g., \cites{ConstantinKiselevEA08,Zlatos10,CotiZelatiDelgadinoEA18,FengIyer19}).  But our results apply to general flows, and the simplicity of the relevant hypotheses allows us to apply them to a wide  class of non-linear parabolic PDEs.

\subsection*{Suppression of blow-up in general non-linear PDEs}

In order to keep the presentation simple, we restrict our attention to the spatially periodic setting.
We hence study the PDE
\begin{equation}\label{e:theta}
  \partial_t \theta + u \cdot \grad \theta = \lap \theta + N(\theta)\,
\end{equation}
on the $d$-dimensional torus $\T^d$ and with initial data $\theta_0\in L^2_0(\T^d)$, where $L^p_0(\T^d)$ is the space of all mean-zero functions in $L^p(\T^d)$.
The velocity $u$ in the convection term is a prescribed (i.e.\ independent of~$\theta$) time-dependent Lipschitz divergence-free vector field.
The
non-linear operator $N\colon H^1(\T^d) \to L^2_0(\T^d)$  is measurable, and its target space being $L^2_0(\T^d)$ means that solutions to~\eqref{e:theta} remain mean-zero.  Finally, $N$ satisfies the following crucial hypotheses:

\begin{enumerate}[({H}1)]
  \item\mylabel{A2}{(H1)}
    There exists $\epsilon_0 \in (0, 1]$ and 
    an increasing continuous function $F \colon [0, \infty) \to [0, \infty)$ such that for every $\varphi \in  H^1(\T^d)$ we have
    \begin{equation*}
      \abs[\Big]{ \int_{\T^d} \varphi N(\varphi) \, dx }
	\leq (1 - \epsilon_0) \norm{\grad \varphi}_{L^2}^2 + F(\norm{\varphi}_{L^2}).
    \end{equation*}
  \item\mylabel{A3}{(H2)}
    There exists $C_0<\infty$ and 
    an increasing continuous function $G \colon [0, \infty) \to [0, \infty)$ such that for every $\varphi \in  H^1(\T^d)$ we have
    \begin{equation*}
      \norm{N(\varphi)}_{L^2} \leq C_0 \norm{\grad \varphi}_{L^2}^2 + G( \norm{\varphi}_{L^2} )\,.
    \end{equation*}
\end{enumerate}


Under hypotheses~\ref{A2}--\ref{A3}, our main result shows that for any mean-zero  initial data $\theta_0$, the corresponding solution to~\eqref{e:theta} is uniformly bounded in $L^2(\T^d)$ on its time interval of existence, provided the \emph{dissipation time} of $u$ is sufficiently small.
The latter is a measure of the dissipation rate enhancement provided by the convection term, defined as follows (see also~\cites{FannjiangWoowski03,Zlatos10,FengIyer19}).

\begin{definition}
\label{d:dissipationTime}
  Let 
  $u \in L^\infty( (0, \infty) \times \T^d)$
  be a divergence-free vector field, and let $\mathcal S_{s, t}$ be the solution operator to the advection-diffusion equation
  \begin{equation}\label{e:ad}
    \partial_t \varphi + u \cdot \grad \varphi = \lap \varphi \,
  \end{equation}
on $(0, \infty) \times \T^d$.
  That is, $\varphi_t \defeq \mathcal S_{s, t} f$ solves~\eqref{e:ad} with initial data $\varphi_s = f$.
 The \emph{dissipation time} of~$u$ is
  \begin{equation*}
    \tau_*(u) \defeq  \inf \set[\Big]{ t \geq 0 \st  \norm{ \mathcal S_{s, s+t} }_{L^2_0 \to L^2_0 } \leq \frac{1}{2} \text{ for all $s\geq 0$} }\,.
  \end{equation*}
\end{definition}

Here, and throughout this paper, $\varphi_t \defeq \varphi(t,\cdot )$ denotes the slice of the function $\varphi$ at time $t$.
Note also that since the $L^2$-norm of solutions to~\eqref{e:ad} is non-decreasing,  $\norm{ \mathcal S_{s, s+t} }_{L^2_0 \to L^2_0 } \leq \frac{1}{2}$ implies $\norm{ \mathcal S_{s, s+t'} }_{L^2_0 \to L^2_0 } \leq \frac{1}{2}$ for all $t'\geq t$.

We can now state our first main result.  
Recall that $\theta$ satisfying~\eqref{e:thetaSpace} is a {\it mild solution} to~\eqref{e:theta} with initial data $\theta_0$ if
  \begin{equation*}
    \theta_t = \mathcal S_{0, t} \theta_0
	+ \int_0^{t} \mathcal S_{s, t}  N(\theta_s) \, ds\,.
  \end{equation*}
We note that mild solutions are also weak (see Section \ref{s:bounds} below).

\begin{theorem}\label{t:ThetaGrowth}
Assume that $N$ satisfies hypotheses~\ref{A2}--\ref{A3} and
  \begin{equation}\label{e:thetaSpace}
    \theta\in L^2_{\mathrm{loc}}((0,T), H^1(\T^d))\cap C([0,T),L^2_0(\T^d))\,
  \end{equation}
  is a  mild solution to~\eqref{e:theta} with
  $u \in L^\infty( (0, \infty) \times \T^d )$
  a divergence-free vector field.
  There is $\tau_0 = \tau_0(\norm{\theta_0}_{L^2}, N)$ such that if $\tau_*(u) \leq \tau_0$, then
  \begin{equation*}
 \sup_{t\in [0,T)} \norm{\theta_t}_{L^2} \leq 2 \norm{\theta_0}_{L^2} + 1 \,.
	  \end{equation*}
  If the functions~$F$ and~$G$ from hypotheses~\ref{A2} and \ref{A3} also satisfy
  \begin{equation}\label{e:FGvanish}
    \limsup_{y\to 0^+} \frac{\sqrt{F(y)}+G(y)}y <\infty\,
  \end{equation}
  and $T = \infty$, then $\norm{\theta_t}_{L^2} \to 0$ exponentially as $t \to \infty$.
\end{theorem}

\begin{remark*}
While~\eqref{e:FGvanish} guarantees $N(0)=0$,  \ref{A2}--\ref{A3} alone do not.  One therefore cannot expect $\lim_{t\to\infty} \theta_t = 0$ in general without assuming~\eqref{e:FGvanish}.
\end{remark*}

We note that the time $\tau_0$ can be computed explicitly in terms of~$F$, $G$, $C_0$, $\epsilon_0$ and $\norm{\theta_0}_{L^2}$.
It is given by $T_0(\norm{\theta_0}_{L^2})$ from Proposition~\ref{p:L2contraction} below for the first claim, and by $T_1(\norm{\theta_0}_{L^2})$ from Proposition~\ref{p:L2decay} below for the second claim.
A lower bound on the exponential decay rate when \eqref{e:FGvanish} holds can also be obtained from Proposition~\ref{p:L2decay}, and is discussed in Remark~\ref{r:decayRate} below.

Finally, this result and its proof easily extend to the PDEs obtained from~\eqref{e:theta} and~\eqref{e:ad} by replacing $\Delta$ by $-(-\Delta)^\gamma$ for any $\gamma>0$, provided one also replaces $\dot H^{\pm 1}(\T^d)$ by $\dot H^{\pm \gamma}(\T^d)$, and $\norm{\grad \varphi}_{L^2}$ in~\ref{A2} and~\ref{A3} by~$\norm{(-\lap)^{\gamma / 2} \varphi}_{L^2}$.

\subsection*{Cellular flows with small dissipation times.}

In order to apply Theorem~\ref{t:ThetaGrowth}, we need to construct flows with arbitrarily small dissipation times.
One construction of such flows is through appropriate rescaling of mixing flows.
Indeed, the action of mixing flows transfers energy from lower to higher frequencies;
the faster this happens, the faster solutions to~\eqref{e:ad} dissipate energy and the smaller~$\tau_*(u)$ becomes.
This principle has previously been used to obtain rigorous bounds on the dissipation time.
In~\cites{ConstantinKiselevEA08} (see also~\cite{BerestyckiHamelEA05,KiselevShterenbergEA08})  the authors show that for time-independent velocity fields~$u=u(x)$ we have $\lim_{A\to\infty} \tau_*(A u) = 0$ if and only if the operator $u \cdot \grad$ on $\T^d$ has no eigenfunctions in $H^1(\T^d)$ other than constants.
In particular, if the flow of~$u$ is weakly mixing (i.e., the spectrum of $u \cdot \grad$ is continuous), then $\tau_*(Au)$ must vanish as the flow amplitude   $A \to \infty$.
Moreover, in~\cites{FengIyer19,CotiZelatiDelgadinoEA18} the authors obtain explicit bounds on $\tau_*(u)$ from the (appropriately defined) \emph{mixing rate} of $u$.
As a result, if $u$ generates an exponentially mixing flow, then $\tau_*( A u(At,\cdot)) \leq \frac cA (\ln A)^2$
for some constant $c > 0$.

The disadvantage of constructing flows with small dissipation times in this manner is that known examples of strongly mixing flows are either quite complicated or not very regular (see for instance~\cite{ConstantinKiselevEA08,YaoZlatos17,ElgindiZlatos19,AlbertiCrippaEA19}).
There are, however, many flows that are far from mixing in any sense but still have small dissipation times.
While these times cannot vanish as the amplitude of the flow is increased, such flows are still sufficient for our purposes.

Here we construct flows with arbitrarily small dissipation times by rescaling
a general class of smooth (time-independent) cellular flows.
A prototypical example of a 2D cellular flow is given by
\begin{equation}\label{e:sinflow}
  u(x)
    = \nabla^\perp \sin(2\pi x_1)\sin(2\pi x_2)
    = 2\pi \begin{pmatrix} -\sin( 2\pi x_1) \cos( 2\pi x_2)\\
      \cos(2\pi x_1) \sin( 2\pi x_2 )
    \end{pmatrix}\,.
\end{equation}
In two dimensions, all cellular flows have closed trajectories and are therefore not mixing.

Nevertheless, we will show that by rescaling both the cell size and the flow amplitude, the dissipation time of such flows can be made arbitrarily small.
%
We will achieve this by establishing a relation between their dissipation time and their (direction-dependent) \emph{effective diffusivity}.
The latter, denoted  $D_e(u)$ for any $e\in\mathbb S^{d-1}$, is the asymptotic long-time mean square displacement in direction~$e$ of the stochastic process on $\R^d$ associated to the operator $\lap-u\cdot\nabla$, 
normalized by the factor $\frac 1{2t}$
(a precise definition can be found at beginning of Section~\ref{s:dtimeBounds} below).
If we let $D(u) \defeq \min\set{D_{e_1}(u), \dots, D_{e_n}(u)}$ be the minimum of the effective diffusivities of the flow $u$ in each of the coordinate directions, then we have the following result.



\begin{theorem}\label{C.7.6}
  For each $n\in\mathbb N$, let $u_n\in W^{1, \infty}(\T^d)$ be a mean-zero divergence-free vector field that is symmetric  in all coordinates (i.e., it satisfies~\eqref{7.6} below).
  If $\lim_{n\to\infty} D(u_n) = \infty$, then there exists $\nu_n\in\mathbb N$ such that the rescaled velocity fields $v_n(x) \defeq -\nu_n u_n( \nu_n x )$ on $\T^d$ satisfy
  \begin{equation*}
    \lim_{n \to \infty} \tau_*(v_n)  = 0\,.
  \end{equation*}
  In particular, for $d \in \set{2, 3}$ and any $T > 0$, there exists a  smooth cellular flow $u$ on $\T^d$ such that $\tau_*(u) \leq T$.
\end{theorem}

We will in fact provide an upper bound on the minimal required  $\nu_n$ (with cell size of $v_n$ being $\nu_n^{-1}$ times the cell size of $u_n$), and further show that with our choice of $\nu_n$ one has
\begin{equation*}
  \tau_*(v_n) \leq C \,D(u_n)^{-\alpha} \ln( 1 + D(u_n) )\,,
\end{equation*}
for some constant~$C$ and some explicit $\alpha > 0$.
The precise details are in Remark~\ref{r:tauDSize} in Section~\ref{s:dtimeBounds} below.

We note that  effective diffusivity of cellular flows has been extensively studied by many authors (see for instance~\cites{ChildressSoward89,FannjiangPapanicolaou94,Koralov04,RyzhikZlatos07}), particularly in two dimensions. 
In fact, typical 2D cellular flows (including the one in~\eqref{e:sinflow}) satisfy $D(Au) \sim A^{1/2}$ as $A\to\infty$ (see Example~\ref{E.7.7} below).
While the asymptotic behavior of~$D(Au)$  is not as well understood in three dimensions, a large class of 3D cellular flows  still has~$\lim_{A\to\infty} D(Au) = \infty$ (see, e.g., Example~\ref{E.7.8} below).
Thus, in both two and three dimensions, we can apply Theorem~\ref{C.7.6} by choosing $u_n(x) = n u(n x)$ for some cellular flow~$u$.


We end this introduction by discussing applications of the above ideas and results to specific non-linear models.

\subsection*{Suppression of blow-up in the Keller-Segel system.}

Chemotaxis is the movement of organisms in response to chemical stimuli, and arises in many contexts such as the movement of bacteria towards food sources and sperm towards eggs, as well as migration of neurons and leukocytes.
The mathematical study of chemotaxis was initiated by Patlak~\cite{Patlak53}, and Keller and Segel~\cites{KellerSegel70,KellerSegel71} who modelled the process as a coupled parabolic system.
Here, we study a simplified, parabolic-elliptic version of this system introduced by J\"ager and Luckhaus~\cite{JagerLuckhaus92} (see also~\cites{Horstmann03,Horstmann04,Perthame07}).
If~$\rho\geq 0$ represents the bacterial population density and~$c\geq 0$ represents the concentration of a chemoattractant  produced by the bacteria, then the evolution of~$\rho$ and~$c$ is governed by
\begin{subequations}
\begin{gather}
  \label{e:ks1}
  \partial_t \rho - \lap \rho = -\dv \paren[\big]{ \rho \chi \grad c }\,,
  \\
  \label{e:ks2}
  - \lap c = \rho - \bar \rho\,,
  \\
  \label{e:ks3}
  \bar \rho = \int_{\T^d} \rho \, dx\,,
\end{gather}
\end{subequations}
which we consider on  $\T^d$.
Here $\chi > 0$ is a sensitivity parameter and we will assume that the dimension $d$ is either $2$ or $3$.
This model stipulates that bacterial diffusion is biased in the direction of the gradient of the concentration of a chemoattractant that is emitted by the bacteria themselves, and that the chemoattractant diffuses much faster than the bacteria do.

It is proved in~\cite{JagerLuckhaus92} that if $d = 2$ and the initial data is below a certain critical threshold, solutions to~\eqref{e:ks1}--\eqref{e:ks3} are regular for all positive time.
Above this threshold, \cite{JagerLuckhaus92} constructs solutions that blow up in finite time by concentrating positive mass at a single point.
In three dimensions, similar results were proved by Herrero et\ al.~\cite{HerreroMedinaEA97,HerreroMedinaEA98,Herrero00} (see also~\cites{Horstmann03,Perthame07}).

We now consider the Keller-Segel system in the presence of a drift, generated by the movement of the ambient fluid.
If $u$ is a divergence-free  vector field representing the fluid velocity, then~\eqref{e:ks1} is replaced by
\begin{gather}\tag{\ref{e:ks1}$'$}\label{e:ks1Drift}
  \partial_t \rho + u \cdot \grad \rho = \lap \rho  -\dv \paren[\big]{ \rho \chi \grad c }\,.
\end{gather}
This model was studied previously by Kiselev and Xu in~\cite{KiselevXu16}, where they show that for any initial population distribution, there exists an ambient velocity field~$u$ that ensures the solution to~\eqref{e:ks1Drift} remains regular for all positive time.
They prove this showing that two families of velocity fields~$u$, the relaxation-enhancing flows from \cite{ConstantinKiselevEA08} and the (initial-data-dependent and time-dependent) exponentially mixing flows from \cite{YaoZlatos17},  can be used to prevent a blow-up in solutions to~\eqref{e:ks1Drift}.

Both these types of flows have fairly complicated geometries.
In contrast, our results above allow us to prove a  result similar to that in~\cite{KiselevXu16}, using general flows that have sufficiently small dissipation times.
This of course includes the flows considered in~\cite{KiselevXu16}, but also the much simpler  fast cellular flows with small cells.
 

\begin{theorem}\label{t:KS}
  Let $d \in \set{2, 3}$ and $\rho_0 \in C^{\infty}(\mathbb{T}^d)$ be any non-negative function.
  There exists $\tau_0 = \tau_0(\norm{\rho_0}_{L^2},\chi) > 0$, such that if 
  for some divergence-free Lipschitz $u$ we have $\tau_*(u)\leq \tau_0$, then the unique solution $\rho$ to~\eqref{e:ks1Drift}, \eqref{e:ks2}, \eqref{e:ks3} with initial data $\rho_0$ is globally regular and 
  \begin{equation}\label{4.2}
  \lim_{t\to\infty} \norm{\rho_t - \bar \rho}_{L^2}=0\,.
  \end{equation}
  In particular, for each such $\rho_0$, there is a time-independent smooth cellular flow $u$ on $\T^d$ that prevents singularity formation  in~\eqref{e:ks1Drift}, \eqref{e:ks2}, \eqref{e:ks3}.
\end{theorem}


We note that there are other flows with a simple structure and large amplitudes (but not small dissipation times), namely shear flows, that have been used to a similar effect in the recent work of Bedrossian and He on \eqref{e:ks1Drift}, \eqref{e:ks2}, \eqref{e:ks3} \cite{BedrossianHe17}.  Obviously, our results apply to general flows with small dissipation times, not just the cellular ones, as well as to more general models than Keller-Segel.  But even when it comes to the Keller-Segel model, it is noteworthy to mention certain differences between blow-up supression via cellular and shear flows.  Strong shear flows quickly extend any confined regions with high initial densities into channels that stretch throughout the domain, and essentially remove one dimension from the dynamics (note that the flows on $\mathbb T^2$ from \cite{KiselevXu16} essentially remove both dimensions).  Thanks to this, \cite{BedrossianHe17} was able to obtain no blow-up for general solutions on $\mathbb T^2$, but such a result on $\mathbb T^3$ only holds for initial data with  small enough total mass.

Since streamlines of cellular flows with small cell sizes are very confined, it is more difficult to obtain fast spreading of confined  high-density regions (this requires the advection and diffusion to cooperate),  but then this fast spreading happens throughout the whole domain (as in \cite{KiselevXu16} on $\mathbb T^2$), rather than only across quasi-one-dimensional channels.  This allows us to prevent blow-up for general initial data on $\mathbb T^3$ as well.  

In addition, while we state all our results on the domain $\mathbb T^d$, it is not too difficult to extend them to cubes $[0,1]^d$ (with homogeneous Neumann or Dirichlet boundary conditions), with the same cellular flows providing the necessary small dissipation times (note that such an extension seems not possible for the flows on $\mathbb T^2$ from \cite{KiselevXu16}, or the shear flows on $\mathbb T^d$).  This is important for applications in real-world settings because unlike $\mathbb T^d$, the cube $[0,1]^d$ is actually a subset of $\mathbb R^d$.

%

It is also worthwhile to point out a difference between Theorem \ref{t:KS} and results on domains of infinite volume, where strong convection may easily increase dissipation by simply spreading even large $L^1$ initial data across large regions (see, e.g., \cite{KiselevZlatos06,HeTadmor19} for such results on the plane involving chemotaxis as well as quenching of reactions from the next subsection).  This is clearly not possible bounded domains, the case considered here.

\subsection*{Quenching in models of combustion.}
Reaction-diffusion equations with ignition type nonlinearities are a well established model used to study the temperature of a combusting  substance.
It is well known that if, on a bounded domain, the average temperature is above the ignition threshold initially, then the fuel eventually burns completely.
However, when the average temperature is below the ignition threshold, the fuel may or may not burn completely, and it is possible for the reaction to be quenched.
We will show in Theorem~\ref{t:rd} in Section~\ref{s:rd} below that, in fact, the reaction can always be quenched in this case by the addition of any convective term with a sufficiently small dissipation time.
The proof does not use Theorem \ref{t:ThetaGrowth} but instead involves showing that $\norm{\mathcal S_{0,T}}_{L^1_0 \to L^\infty_0}$ (for any fixed  $T>0$) can be made arbitrarily small by making the dissipation time of the advecting velocity field small enough.

\subsection*{Organization of the paper.}
We prove Theorem \ref{t:ThetaGrowth} in Section 2, Theorem \ref{t:KS} in Section 3, and discuss applications to fluid dynamics and combustion models in Sections 4 and 5, respectively.  Theorem \ref{C.7.6} is proved in Section 6, along with some results concerning asymptotic mean displacement of stochastic processes corresponding to \eqref{e:ad} on $\R^d$ with periodic flows.

\subsection*{Acknowledgements.}

We thank Alexander Kiselev and Bruce Driver for helpful discussions.
We also thank Theo Drivas, Franco Flandoli and the anonymous referees for pointing out an error in an earlier version of this paper.

\section{Uniform energy bounds for general PDEs} \label{s:bounds}

In this section we prove Theorem~\ref{t:ThetaGrowth},
%
%
which is an immediate corollary of the following two results.  

\begin{proposition}\label{p:L2contraction}
Assume that~$N$ satisfies hypotheses~\ref{A2}--\ref{A3},
$u \in L^\infty( [0, \infty) \times \T^d )$
is a divergence-free vector field, and~$\theta$ is a  mild solution to~\eqref{e:theta} such that~\eqref{e:thetaSpace} holds.
  For any $B \geq 0$, let 
  \begin{equation*}
    T_0(B) \defeq \min\set[\Big]{
      \int_{B}^{2B + 1} \frac{y}{ F(y)} \, dy\,,\ 
      \frac{ B}{2C_0 \epsilon_0^{-1}F(2B + 1) + 2  G(2B + 1 ) } }\,.
  \end{equation*}
  If  for some $t_0 \geq 0$ we have $\tau_*(u) \leq T_0(\norm{\theta_{t_0}}_{L^2})$, 
  then
  \begin{subequations}
  \begin{gather}
    \label{e:L2decay1}
      \norm{\theta_{t_0 + n \tau_*(u)})}_{L^2} \leq \norm{\theta_{t_0}}_{L^2}
      \quad\text{for every } n \in \N\,,
      \\
    \label{e:L2decay2}
	\norm{\theta_t}_{L^2} \leq 2 \norm{\theta_{t_0}}_{L^2} + 1
	\quad\text{for every } t > t_0\,.
  \end{gather}
  \end{subequations}
\end{proposition}



\begin{proposition}\label{p:L2decay}
Assume the hypotheses of Proposition \ref{p:L2contraction} as well as \eqref{e:FGvanish}.
  Let $T_1(0)\defeq \infty$, and for any $B >0$ let
    \begin{equation*}
    T_1(B) \defeq \inf_{b\in(0,B]} \min\set[\Big]{
     \int_{b}^{2 b} \frac{y}{F(y)} \, dy,\ 
      \frac{b}{4C_0\epsilon_0^{-1} F(2 b) + 4  G(2b) } }  \,.
  \end{equation*}
  If  for some $t_0 \geq 0$ we have  $\tau_*(u) \leq T_1(\norm{\theta_{t_0}}_{L^2})$, then
  \begin{equation}\label{e:edecay}
    \norm{\theta_{ t}}_{L^2} \leq  2\Psi^{\lfloor (t-t_0) / \tau_*(u) \rfloor} \left( \norm{\theta_{t_0}}_{L^2} \right)
    \quad\text{for every } t \geq t_0\,,
  \end{equation}
  where $\Psi(B)\defeq B-  \min \{\frac B{16},\frac{\epsilon_0}{8C_0} \} $ and $\Psi^n\defeq \Psi\circ\Psi^{n-1}$ for $n\geq 2$.
\end{proposition}

\begin{remark*}
Hypothesis~\eqref{e:FGvanish} ensures that $T_1(B)>0$ for every $B>0$, because  $\int_{b}^{2 b} \frac{y}{y^\alpha} \, dy$ is uniformly bounded in all small $b>0$ when $\alpha\leq 2$.
\end{remark*}

\begin{remark}\label{r:decayRate}
The time decay obtained in~\eqref{e:edecay} is exponential.
The exponential decay rate is $\tau_*(u)^{-1}|\ln (1-r)| $, where $r$ is initially
\begin{equation*}
  \min \set[\Big]{ \frac 1 {16},\frac{\epsilon_0}{8C_0 \norm{\theta_{t_0}}_{L^2}} }\,,
\end{equation*}
and once $t$ becomes large enough, $r$ changes to $\frac 1{16}$.
By optimizing the proof further, this number can be increased to any number smaller than
\begin{equation*}
\frac 1{2}-\lim_{B\to0+} \frac {\tau_*(u)}{4T_1(B)}\,.
\end{equation*}
\end{remark}
\smallskip


Before proving these results, let us note
%
%
%
that our mild solutions are also weak solutions.  We say that $\theta$ is a {\it weak solution} to~\eqref{e:theta} on $(0,T)\times\T^d$ if it is a distributional solution to the PDE, and  the equality in~\eqref{e:theta} holds in $H^{-1}(\T^d)$ at almost every $t\in(0,T)$.

If now $\theta$ as in \eqref{e:thetaSpace} is a mild solution to~\eqref{e:theta} on $(0,T)\times\T^d$, then standard smoothing properties of $S_{s,t}$ when $t-s$ is uniformly positive show that
\[
    \theta_t^\delta \defeq \mathcal S_{0, t} \theta_0
	+ \int_0^{t-\delta} \mathcal S_{s, t}  N(\theta_s) \, ds
\]
is a strong and therefore also distributional solution to~\eqref{e:theta} with $N(\theta_t)$ replaced by $S_{t-\delta,t} N(\theta_{t-\delta})$, on $(\delta,T)\times\T^d$.  
Taking $\delta \to 0$ in the definition of distributional solutions, and using \eqref{e:thetaSpace}, \ref{A3}, and measurability of $N$ (which yield $N(\theta)\in L^1_{\mathrm{loc}}((0,T), L^2(\T^d))$), as well as $\|S_{s,t}\|_{L^p\to L^p}\le 1$, now shows that  $\theta$ is a distributional solution to~\eqref{e:theta} on $(0,T)\times\T^d$.

Since \eqref{e:thetaSpace} and $u \in L^\infty( (0, \infty) \times \T^d)$ yield $\Delta\theta\in L^2_{\mathrm{loc}}((0,T), H^{-1}(\T^d))$ and $u\cdot\nabla\theta,N(\theta)\in L^1_{\mathrm{loc}}((0,T), L^2(\T^d))$, it follows that
\begin{equation}
\label{2.1a}
\partial_t  \theta 
 \in L^2_{\mathrm{loc}}((0,T), H^{-1}(\T^d)) + L^1_{\mathrm{loc}}((0,T), L^2(\T^d))\,.
\end{equation}
Hence, the equality in~\eqref{e:theta} holds at almost every $t\in(0,T)$ in $H^{-1}(\T^d)$, so $\theta$ is indeed a weak solution to~\eqref{e:theta} on $(0,T)\times\T^d$.

\begin{proof}[Proof of Proposition~\ref{p:L2contraction}]
  Without loss of generality, we assume $t_0 = 0$.
  For notational convenience let $B\defeq \norm{\theta_{0}}_{L^2}$ and $\tau_*=\tau_*(u)$.


  Let us now assume that $\norm{\theta_{s}}_{L^2}\leq B$ for some $s\geq 0$ (which holds for $s=0$).
  Note that similarly to~\cite[Theorem~5.9.3]{Evans98}, one can use \eqref{2.1a} and \eqref{e:thetaSpace} to show that~$\norm{\theta_t}_{L^2}^2$ is a locally absolutely continuous function of~$t$, whose derivative equals  $2\int_{\T^d} \partial_t\theta_t \, \theta_t \,dx$ (the integral representing a pairing of elements from $H^{-1}(\T^d)$ and $H^1(\T^d)$) at almost every $t\in(0,T)$.
  Hence, multiplying~\eqref{e:theta} by~$\theta_t\in H^1(\T^d)$ and integrating in space yields
  \begin{align}
    \nonumber
    \frac{1}{2} \partial_t \norm{\theta_t}_{L^2}^2
      + \norm{\grad \theta_t}_{L^2}^2
      &\leq \abs[\Big]{ \int_{\T^d} \theta_t N(\theta_t) \, dx }
    \\
    \label{e:eeN}
      &\leq (1 - \epsilon_0) \norm{\grad \theta_t}_{L^2}^2
	+ F(\norm{\theta_t}_{L^2})\,
  \end{align}
for almost every $t\in(0,T)$,  so for these $t$ we have
  \begin{equation}\label{e:ee1}
     \partial_t \norm{\theta_t}_{L^2} \norm{\theta_t}_{L^2} \leq F( \norm{\theta_t}_{L^2} )\,.
  \end{equation}
  Thus, for almost every~$t\in(0,T)$ we have 
  \[
    \partial_t \int_{B}^{\norm{\theta_t}_{L^2}} \frac y{F(y)} dy = \frac {\partial_t \norm{\theta_t}_{L^2} \norm{\theta_t}_{L^2}} { F( \norm{\theta_t}_{L^2} )} \leq 1.
  \]
Since $\int_{B}^{\norm{\theta_s}_{L^2}} \frac y{F(y)} dy\leq 0$, it follows that for all $t\in [s,s+T_0(B)]$ we have
     \begin{equation}\label{e:doublingTime}
 \int_{B}^{\norm{\theta_t}_{L^2}} \frac y{F(y)} dy \leq 
   t-s \leq T_0(B) \leq 
      \int_{B}^{2B + 1} \frac{y}{ F(y)} \, dy \,,
  \end{equation}
   which forces $\norm{\theta_t}_{L^2}\leq 2B+1$ for all $t\in [s,s+T_0(B)]$.

  Next, integrating~\eqref{e:eeN} in time yields
  \begin{equation}\label{e:eeInt}
    \norm{\theta_{s+\tau_*}}_{L^2}^2 - \norm{\theta_s}_{L^2}^2
      \leq 2 \int_s^{s+\tau_*} F(\norm{\theta_t}_{L^2}) \, dt
	- 2\epsilon_0 \int_s^{s+\tau_*} \norm{\grad \theta_t}_{L^2}^2 \, dt\,.
  \end{equation}
  If
  \begin{equation}\label{e:h1LargeN}
    \int_s^{s+\tau_*}
      \norm{\grad \theta_t}_{L^2}^2 \, dt
	\geq \frac{1}{\epsilon_0} \int_s^{s+\tau_*} F(\norm{\theta_t}_{L^2}) \, dt\,,
  \end{equation}
  then the right hand side of~\eqref{e:eeInt} is at most~$0$, and hence
$\norm{\theta_{s+\tau_*}}_{L^2}\leq B$.
  If \eqref{e:h1LargeN} fails, then
from
  \begin{equation*}
    \theta_{ s+\tau_*}
      = \mathcal S_{s, s+\tau_*} \theta_s
	+ \int_s^{s+\tau_*} \mathcal S_{t, s+\tau_*}  N(\theta_t) \, dt\,
  \end{equation*}
we obtain
  \begin{align}
    \nonumber
    \norm{\theta_{s+\tau_*}}_{L^2}
      &\leq
	\norm{ \mathcal S_{s, s+\tau_*} \theta_s }_{L^2}
	+ \int_s^{s+\tau_*} \norm[\big]{ \mathcal S_{t, s+\tau_*} N(\theta_t) }_{L^2} \, dt
    \\
    \nonumber
      &\leq \frac{B}{2} + \int_s^{s+\tau_*} \norm{N( \theta_t)}_{L^2}  \, dt
    \\
    \nonumber
      &\leq
	\frac{B}{2} + \int_s^{s+\tau_*} \paren[\big]{ C_0 \norm{\grad \theta_t}_{L^2}^2 + G(\norm{\theta_t}_{L^2} ) }\, dt
	\\
	    \label{e:gi5N}
	&\leq
	\frac{B}{2} + \int_s^{s+\tau_*} \paren[\Big]{\frac{C_0}{\epsilon_0} F(\norm{\theta_t}_{L^2}) + G(\norm{\theta_t}_{L^2} ) }\, dt\,.
  \end{align}
Since $\tau_*\leq T_0(B)$ (so that $\norm{\theta_t}_{L^2}\leq 2B+1$ for all $t$ in the integral), it follows that
  \begin{equation*}
    \norm{\theta_{s+\tau_*}}_{L^2}
      \leq \frac{B}{2} +\paren[\Big]{\frac{C_0}{\epsilon_0} F(2B + 1) + G(2B + 1) }  \tau_* \,\leq B.
  \end{equation*}
  
From this and $\norm{\theta_t}_{L^2}\leq 2B+1$ for all $t\in [s,s+\tau_*]$ both holding whenever $\norm{\theta_s}_{L^2}\leq B$ (in particular, when $s=0$), the claim follows via induction on $n$.
\end{proof}


\begin{proof}[Proof of Proposition~\ref{p:L2decay}]
  Without loss of generality, we again assume $t_0 = 0$.
  We again also denote $B\defeq \norm{\theta_{0}}_{L^2}$ and  $\tau_*=\tau_*(u)$.
  By 
  \eqref{e:FGvanish} and continuity of $G$ we have $G(0)=0$, therefore $N(0)=0$ by  \ref{A3}.
  This proves the claim when $\theta_0=0$, so let us now assume that $B>0$.
  
  As in the previous proof, from~\eqref{e:ee1} we obtain
  \begin{equation}\label{e:doubling1}
    \norm{\theta_t}_{L^2} \leq 2 \norm{\theta_s}_{L^2}
    \quad\text{whenever}\quad
      t\in \left[s,
	s+\int_{\norm{\theta_s}_{L^2}}^{2 \norm{\theta_s}_{L^2}}
	  \frac{y}{F(y)} \, dy \right] \,.
  \end{equation}

  Next, let $C_B\defeq \max\{C_0B,2\epsilon_0\}$.  From~\eqref{e:eeInt} we see that if
  \begin{equation}\label{e:h1LargeN1}
    \int_0^{\tau_*}
      \norm{\grad \theta_t}_{L^2}^2 \, dt
	\geq \frac{1}{\epsilon_0} \int_0^{\tau_*} F(\norm{\theta_t}_{L^2}) \, dt
	  + \frac{C_B-\epsilon_0}{4 C_B^2} B^2\,,
  \end{equation}
  then we obtain (also using $\epsilon_0\leq C_B/2$ and the definition of $C_B$)
  \begin{align}
    \nonumber
    \norm{\theta_{\tau_*}}_{L^2}
      &\leq \paren[\Big]{1-\frac{\epsilon_0(C_B-\epsilon_0)}{2 C_B^2} }^{1/2} B
    \\
    \label{e:dbl1}
    &\leq \paren[\Big]{ 1- \frac{\epsilon_0}{8 C_B} } B
    \leq \paren[\Big]{ 1-
      \min\set[\Big]{\frac 1{16} ,\frac{\epsilon_0}{8C_0B} } }
      B \,.
  \end{align}
  And if \eqref{e:h1LargeN1} fails, then the argument proving \eqref{e:gi5N} shows that
  \begin{equation*}
    \norm{\theta_{\tau_*}}_{L^2} \leq
	\frac{B}{2} + \int_0^{\tau_*} \paren[\Big]{\frac{C_0}{\epsilon_0} F(\norm{\theta_t}_{L^2}) + G(\norm{\theta_t}_{L^2} ) }\, dt + \frac{C_0(C_B-\epsilon_0)}{4 C_B^2} B^2\,.
  \end{equation*}
  From~\eqref{e:doubling1} and $F,G$ being increasing, we see that if $\tau_* \leq T_1(B)$, then the integral is bounded above by $\frac B4$ and so
  \begin{equation*}
    \norm{\theta_{\tau_*}}_{L^2} 
    	\leq \left( \frac 34  +\frac{C_0C_B-\epsilon_0}{4 C_B^2} B \right) B
	 \leq \left( 1-  \min\left\{\frac 1{8},\frac{\epsilon_0}{4C_0B} \right\} \right) B\,.
  \end{equation*}
  
Applying this argument iteratively on time intervals $[(n-1)\tau_*,n\tau_*]$ for $n\in\N$, and with $B$ replaced by $\norm{\theta_{(n-1)\tau_*}}_{L^2}$, shows that
  \begin{equation}\label{e:edecay1}
    \norm{\theta_{n \tau_*}}_{L^2} \leq 
    \Psi^n(\norm{\theta_0}_{L^2}) 
  \end{equation}
  for all $n \in \N$.
  From~\eqref{e:doubling1} we know that $\norm{\theta_t}_{L^2}$ at most doubles at times between integer multiples of $\tau_*$.
  Combining this fact with~\eqref{e:edecay1} yields~\eqref{e:edecay}, as desired.
\end{proof}

\section{Preventing blow-up of Keller-Segel dynamics}

In this section we prove Theorem \ref{t:KS}.
For the sake of convenience, we denote
\begin{equation*}
  \grad^{-1} \defeq \grad \lap^{-1}\,.
\end{equation*}
That is, we let $\grad^{-1} \psi \defeq \grad \phi$, where $\phi$ solves $\lap \phi = \psi$ on $\T^d$.
With this notation, the first two equations of the Keller-Segel system~\eqref{e:ks1Drift}, \eqref{e:ks2}, \eqref{e:ks3} read
\begin{equation}\label{e:chemo1}
  \partial_t\rho + u\cdot\grad \rho = \lap \rho
    + \chi \grad\cdot(\rho \grad^{-1}(\rho-\bar{\rho}))\,.
\end{equation}


The main idea here is to use the results from Section \ref{s:bounds} to show that $\|\rho_t\|_{L^2}$ remains uniformly bounded in time.  
Once this is established, well known results can be used to obtain global regularity, and then~\eqref{4.2} will be obtained from \eqref{e:edecay}.

\begin{lemma}\label{l:KSL2}
  Let $\rho$ be a smooth solution to~\eqref{e:ks1Drift}, \eqref{e:ks2}, \eqref{e:ks3}  on the time interval $[0, T)$, and let $t_0\in [0,T)$.
  There are  $c = c(\chi\norm{\rho_{t_0} - \bar \rho}_{L^2}) > 0$ and $\tau_1 = \tau_1(\norm{\rho_{t_0} - \bar \rho}_{L^2}, \bar \rho,\chi) > 0$
  such that if $\tau_* (u) \leq \tau_1$ for the incompressible drift $u$,
  then 
  \begin{equation}\label{e:ks1L2decay}
    \norm{\rho_t - \bar \rho}_{L^2} \leq 3 e^{-c (t - t_0)/\tau_*(u)} \norm{\rho_{t_0} - \bar \rho}_{L^2}\,,
  \end{equation}
  for all $t \in [t_0, T)$.
  Moreover, $\tau_1$ and $c$ can be chosen to be decreasing in each argument.
\end{lemma}

\begin{proof}
Notice that if we let
  \begin{gather*}
    \theta \defeq \rho - \bar \rho\,,
    \\
    N(\theta) \defeq  \chi \dv \paren[\big]{ (\theta + \bar \rho) \grad^{-1} \theta}\,,
  \end{gather*}
  then~$\theta$ satisfies~\eqref{e:theta} (here we think of $\bar\rho_0$ as a parameter) and $N(\theta)$ is mean-zero.
  Thus, in order to apply Proposition~\ref{p:L2decay}, we only need to verify hypotheses~\ref{A2}--\ref{A3} and \eqref{e:FGvanish}.
  For~\ref{A2}, we compute
  \begin{align}
    \nonumber
    \chi^{-1}\int_{\T^d} \theta N(\theta) \, dx
      &= \int_{\T^d} (\theta + \bar \rho) \theta^2 \, dx
	+  \int_{\T^d} \theta \, \grad \theta \cdot \grad^{-1} \theta  \, dx
    \\
    \label{e:gi3}
      &= \frac{1}{2} \norm{\theta}_{L^3}^3
	+  \bar \rho \norm{\theta}_{L^2}^2
  \end{align}
because
\[
\int_{\T^d} \theta \, \grad \theta \cdot \grad^{-1} \theta  \, dx =
    \frac 12 \int_{\T^d} \grad \theta^2  \cdot \grad^{-1} \theta  \, dx
= -\frac{1}{2} \int_{\T^d} \theta^3  \, dx\,.
\]

  Using the Gagliardo-Nirenberg and Young inequalities (see for instance \cite{Mazja85}), we have
  \begin{equation*}
    \norm{\theta}_{L^3}^3
      \leq C_1 \norm{\theta}_{L^2}^{3-\frac{d}{2}}\norm{\grad \theta}_{L^2}^{\frac{d}{2}}
      \leq \chi^{-1}\norm{\grad \theta}_{L^2}^2
	+C_2 \chi^{\frac{d}{4-d}}_{\vphantom{L^2}} \norm{\theta}_{L^2}^{\frac{12-2d}{4-d}}\,,
  \end{equation*}
  for some universal constants $C_1,C_2$.
  Using this in~\eqref{e:gi3} shows that~\ref{A2} holds with $\epsilon_0=\frac 12$ and 
  \begin{equation*}
    F(y) =
     \left( C_1 \chi^{\frac{4}{4 - d}}
	y^{\frac{4}{4 -d}}
      + \chi \bar \rho \right) y^2 \,.
  \end{equation*}

  For~\ref{A3}, note that the Hardy-Littlewood-Sobolev inequality implies
  \begin{align*}
    \chi^{-1} \norm{N(\theta)}_{L^2}
      &\leq
	\norm{\theta (\theta + \bar \rho)}_{L^2}
	+ \norm{\grad \theta \cdot \grad^{-1} \theta }_{L^2}
    \\
    &\leq
      \bar \rho \norm{\theta}_{L^2}
      + \norm{\theta}_{L^4}^2
      + \norm{\grad\theta}_{L^2}
	\norm{\grad^{-1}\theta}_{L^\infty}
    \\
    &\leq
      \bar \rho \norm{\theta}_{L^2}
      + C_3 \paren[\big]{
	\norm{\theta}_{\dot H^{d/4}}^2
	+ \norm{\grad\theta}_{L^2}
	  \norm{\theta}_{\dot H^{d/2 -1 + \epsilon}}
      }
    \\
    &\leq
      \bar \rho \norm{\theta}_{L^2}
      + C_4 \norm{\grad \theta}_{L^2}^2\,
  \end{align*}
 for any $\epsilon>0$ and constants $C_3,C_4$ depending only on $\epsilon$.
  (Above we also used for $\psi=\lap^{-1}\theta$ that for any $s > 0$ we have $\norm{D^2 \psi}_{\dot H^s} \leq C \norm{\lap\psi}_{\dot H^s}$, with some constant $C= C(s, d)$.)
Hence~\ref{A3} is satisfied with $C_0=C_4\chi $ (after fixing, e.g., $\epsilon=\frac 12$) and $G(y) = \chi \bar \rho y$.

  From the above formulae for $F$ and $G$ we can see that 
  \eqref{e:FGvanish} also holds.
  Thus Proposition~\ref{p:L2decay} applies, proving \eqref{e:ks1L2decay} as well as the last claim.
\end{proof}

Next,  we use the fact that a time-uniform bound on~$\norm{\rho_t}_{L^2}$ implies global existence for~\eqref{e:chemo1}.
This was shown in~\cite{KiselevXu16}, and we present here a short proof for the sake of completeness.

\begin{lemma}[Thm.\ 2.1 in~\cite{KiselevXu16}]
\label{l:L2cr}
  Let $d\in \set{2, 3}$ and $\rho_0 \in C^\infty(\T^d)$ be non-negative. 
  If the maximal time $T$ of existence of the unique sm\-ooth solution $\rho$ to~\eqref{e:chemo1} with  initial data $\rho_0$ and incompressible drift $u$ is finite, then
  \begin{equation}\label{13th}
    \int_0^T\norm{\rho_t -\bar{\rho}}_{L^2}^{\frac{4}{4-d}}\, dt = \infty \,.
  \end{equation}
\end{lemma}
\begin{proof}
Multiplying~\eqref{e:chemo1} by $-\lap \rho$ and integrating in space yields
\begin{multline}\label{e:ksee1}
  \frac{1}{2}\partial_t\norm{\grad \rho}_{L^2}^2
    +\norm{\lap \rho}^2_{L^2}
    \leq \abs[\Big]{\int_{\mathbb{T}^d} \grad\rho \cdot \grad^{-1}(\rho-\bar{\rho}) \lap\rho \,dx } 
  \\
  + \abs[\Big]{\int_{\mathbb{T}^d}u\cdot \grad\rho \, \lap\rho \,dx}
  + \abs[\Big]{\int_{\mathbb{T}^d}\rho(\rho-\bar{\rho})\lap \rho \,dx}
  = \mathit{I} + \mathit{II} + \mathit{III}\,.
\end{multline}
To estimate $\mathit{I}$, we integrate by parts to get
\[
  \mathit{I} \leq
    \frac{1}{2}	\abs[\Big]{  \int_{\T^d}  \nabla \abs{\grad \rho}^2 \cdot \grad^{-1}(\rho-\bar{\rho})  \,dx  }
      + \sum_{i=1}^d
      \abs[\Big]{	  \int_{\mathbb{T}^d}(\grad\rho)\cdot (\grad^{-1}(\rho-\bar{\rho}))_{x_i} \rho_{x_i} \,dx } \,.
\]
Another integration by parts in the first term shows that it is bounded above by 
\[
\int_{\T^d} \abs{\grad \rho}^2 \abs{\rho - \bar \rho} \,dx \leq  \norm{\rho - \bar \rho}_{L^4} \norm{\grad \rho}_{L^{8/3}}^{2}\,
\]
while $L^4$-boundedness of double Riesz transforms shows that the second term is bounded above by $C \norm{\rho - \bar \rho}_{L^4} \norm{\grad \rho}_{L^{8/3}}^{2}$ for some constant $C$.  The Gagliardo-Nirenberg inequality applied to both norms in this product then yields 
\[
I\leq  C    \norm{\grad \rho}_{L^2}^{\frac d4}   \norm{\rho - \bar \rho}_{L^2}^{1 - \frac d4}    
 \norm{\lap \rho}_{L^2}^{\frac d4}  \norm{\grad \rho}_{L^2}^{2-\frac d4}       
 = C  \norm{\grad \rho}_{L^2}^{2}   \norm{\rho - \bar \rho}_{L^2}^{1 - \frac d4}    
 \norm{\lap \rho}_{L^2}^{\frac d4}  
\]
and therefore finally
\[
I\leq C       \norm{\grad \rho}_{L^2}^{2 - \frac d2}   \norm{\rho - \bar \rho}_{L^2}   \norm{\lap \rho}_{L^2}^{\frac d2} 
\leq 
\frac{1}{2} 
\norm{\lap \rho}_{L^2}^2
    + C \norm{\rho - \bar \rho}_{L^2}^{\frac 4{4-d}} \norm{\grad \rho}_{L^2}^2
\]
(with a new $C$ in each inequality).

For~$\mathit{II}$, we again integrate by parts and use $\grad\cdot u=0$ to get
\[
  \mathit{II} \leq 
    \left| \int_{\mathbb{T}^d}u\cdot \grad|\grad\rho|^2 \,dx \right| + C \norm{u}_{C^1}\norm{\grad \rho}^2_{L^2} 
    = C \norm{u}_{C^1}\norm{\grad \rho}^2_{L^2} \,.
\]

For~$\mathit{III}$, we integrate by parts and use the Gagliardo-Nirenberg inequality exactly as we did for term~$\mathit{I}$ to obtain
\begin{align*}
\mathit{III}
  &\leq 2\int_{\mathbb{T}^d}|\rho-\bar\rho| |\grad \rho|^2 \,dx + \bar{\rho} \int_{\mathbb{T}^d}|\grad \rho|^2 \,dx
  \\
  &\leq
    \frac{1}{2} 
  \norm{\lap \rho}_{L^2}^2
    + \left(C \norm{\rho - \bar \rho}_{L^2}^{\frac 4{4-d}}
      +  \bar \rho \right) \norm{\grad \rho}_{L^2}^2\,.
\end{align*}

Using the above bounds in~\eqref{e:ksee1} yields
\begin{equation}
  \partial_t \norm{\grad \rho}_{L^2}^2
    +\norm{\lap \rho}_{L^2}^2 
    \leq C \left( \norm{u}_{C^1} + \bar \rho
	    +  \norm{\rho - \bar \rho}_{L^2}^{\frac 4{4 - d}} \right)
	  \norm{\grad \rho}_{L^2}^2\,.
\end{equation}
Gronwall's lemma now shows that if~\eqref{13th} does not hold, then 
$\norm{\grad \rho}_{L^2}$ and $\norm{\lap \rho}_{L^2}$ remain uniformly bounded on $[0,T)$.
Then a standard Galerkin approximations  argument shows that  $\rho$ can be smoothly extended beyond $T$, a contradiction.
This completes the proof.
\end{proof}

We can now prove Theorem~\ref{t:KS}.

\begin{proof}[Proof of Theorem~\ref{t:KS}]
Let $\tau_0(y,\chi) \defeq \tau_1(y,y,\chi)$, with $\tau_1$ being from Lemma~\ref{l:KSL2} (so $\tau_0$ is also decreasing in both arguments).  
   If $T$ is the maximal time of existence of $\rho$, then~\eqref{e:ks1L2decay} holds (with $t_0 = 0$) for all $t\in [0,T)$.  Therefore we must have $T=\infty$ because otherwise~\eqref{13th} would not hold, contradicting Lemma~\ref{l:L2cr}.
  This also means that \eqref{4.2} follows from \eqref{e:ks1L2decay}.
The last claim follows from Theorem \ref{C.7.6}.
  \end{proof}

\section{Quenching of reactions by stirring}\label{s:rd}

A well established model  for the (normalized) temperature of a combusting substance (see, e.g., the reviews \cite{Berestycki02,Xin00,Xin09}) is the reaction diffusion equation
\begin{equation}\label{e:rd}
  \partial_t \theta = \lap \theta + f(\theta)\,.
\end{equation}
Here, the function $\theta$ takes values in $[0,1]$ and the non-linear Lipschitz {\it reaction function} $f$ is of the \emph{ignition type} (other types of reaction functions model other reactive processes, including chemical kinetics and population dynamics).  That is, we assume that there is $\alpha_0\in(0,1)$ such that
\begin{equation*}
  f(\theta) = 0 \quad\text{for } \theta \in [0, \alpha_0] \cup \set{1}
  \quad\text{and}\quad
  f(\theta) > 0 \quad\text{for } \theta \in (\alpha_0, 1)\,.
\end{equation*}
The number $\alpha_0$ is  {\it ignition temperature}, below which no burning occurs.  Also note that since $\theta\equiv 0,1$ are stationary solutions, comparison principle indeed yields $0\leq\theta\leq 1$ whenever $0\leq\theta_0\leq 1$.

When the combustive process is also subject to mixing due to a prescribed incompressible drift~$u$, equation~\eqref{e:rd} becomes
\begin{gather}\tag{\ref{e:rd}$'$}\label{e:rdDrift}
  \partial_t \theta + u \cdot \grad \theta = \lap \theta + f(\theta)\,.
\end{gather}
With or without the drift~$u$,  one can show using 
\[
  \partial_t \norm{\theta_t}_{L^1} = \int_{\mathbb T^d} f(\theta_t)\,dx \geq 0
\]
that the fuel burns completely (i.e., $\lim_{t\to\infty}\norm{1-\theta_t}_{L^\infty}=0$) for all initial data $0\leq\theta_0\leq 1$ that satisfy $\bar \theta_0 \defeq \int_{\T^d} \theta_0 \, dx \geq \alpha_0$ and $\theta_0\not\equiv\alpha_0$.
%
%
%
%

When instead $\bar \theta_0 < \alpha_0$, the fuel may or may not burn completely, and it is possible for the reaction to be {\it quenched}, that is, $\norm{\theta_T}_{L^\infty} \leq \alpha_0$ for some $T\geq 0$.
Comparison principle shows that after such time $T$,  evolution of the temperature will be governed by the linear equation~\eqref{e:ad}, and hence we will have $\lim_{t\to\infty}\norm{\theta_t-\bar\theta_0}_{L^\infty}=0$.

The main result of this section shows that if $\bar \theta_0 < \alpha_0$, then one can always ensure quenching by choosing an incompressible drift~$u$ with a small enough dissipation time.

\begin{theorem}\label{t:rd}
  Let $\theta$ be the solution to~\eqref{e:rdDrift} on $\T^d$ with nonnegative initial data~$\theta_0 \in L^\infty(\T^d)$ with $\bar \theta_0 < \alpha_0$.  There is $\tau_0=\tau_0(\alpha_0,\bar\theta_0)$ such that if $\tau_*(u) \leq\tau_0$ for some divergence-free vector field $u$, then  the reaction is quenched.
   In particular,  if $\bar \theta_0 < \alpha_0$ and $d\in\{2,3\}$, then there is a time-independent smooth cellular flow $u$ on $\T^d$ that quenches the reaction.
\end{theorem}

\begin{remark*}
 The last claim should also hold for $d\geq 4$, but we are not aware of a construction of such flows.  Time-periodic flows with this property can be constructed by alternating flows from Example \ref{E.7.7} below acting on different pairs of coordinates.
\end{remark*}

Theorem~\ref{t:rd} and its proof are closely related to Theorem 7.2 in \cite{ConstantinKiselevEA08}, which is a qualitative statement about a certain class of time-independent drifts (so-called relaxation-enhancing ones) on general compact manifolds.
Theorem~\ref{t:rd} is a more quantitative result that applies to general time-dependent drifts.  It is an immediate consequence of the comparison principle and the following result about equation~\eqref{e:ad}.


\begin{proposition}\label{p:L1Linf}
There is a  constant $C_d$ such that the solution operator  $\mathcal S_{s, t}$ for~\eqref{e:ad} with any
$u \in L^\infty( (0, \infty), W^{1,\infty}( \T^d) )$
satisfies
\begin{equation}\label{e:l1linf}
  \sup_{s \geq 0} \norm{\mathcal S_{s, s+t}}_{L^1_0 \to L^\infty_0} \leq \epsilon 
\end{equation}
for every $\epsilon>0$, provided
\begin{equation} \label{6.3}
t\geq  (C_d + d\ln^- \tau_*(u) + 2\ln^-\epsilon  ) \tau_*( u )\,.
\end{equation}
\end{proposition}

\begin{proof}
  There is a $u$-independent constant $c_d\geq 1$ such that if $\varphi$ is the solution to~\eqref{e:ad} with mean-zero initial data $\varphi_0 \in L^1(\T^d)$, then
   \beq\lb{6.4}
    \norm{\varphi_{s + t}}_{L^\infty} \leq c_d t^{-d/4} \norm{\varphi_s}_{L^2}
    \eeq
and
   \beq\lb{6.5}
    \norm{\varphi_{s + t}}_{L^2} \leq c_d t^{-d/4} \norm{\varphi_s}_{L^1}
    \eeq
 for any $s\geq 0$ and $t\in[0,1]$.  The first claim is contained in Lemma~5.4 in~\cite{Zlatos10} (see also Lemma~5.6 in \cite{ConstantinKiselevEA08} and Lemmas 3.1, 3.3 in~\cite{FannjiangKiselevEA06}), while the second follows from it and a simple duality argument (it is also contained in Lemma~5.4 in~\cite{Zlatos10} but with $\frac d2$ in place of $\frac d4$, which would also suffice here).   Notice also that while \cite{Zlatos10} only considers time-independent drifts, the proof of its Lemma~5.4 equally applies to the time-dependent case.
 

  Write $\tau_* = \tau_*(u)$ and let $\sigma = \min\{\tau_*,1\}$.  Then for any $s\geq 0$,  $n \in\mathbb N$, and $t\geq n \tau_*+2\sigma$  we have
  \begin{align*}
    \norm{\varphi_{s+t}}_{L^\infty}
      &\leq \norm{\varphi_{s+n \tau_*+2\sigma}}_{L^\infty}
      \leq c_d \sigma^{-d/4} \norm{\varphi_{s+n\tau_* +\sigma }}_{L^2}
      \\
      &\leq c_d \sigma^{-d/4} 2^{-n}\norm{\varphi_{s+\sigma} }_{L^2}
      \leq c_d^2 \sigma^{-d/2} 2^{-n} \norm{\varphi_s }_{L^1}\,.
  \end{align*}
 The result now follows from  this estimate with
  \begin{equation*}
    n =   \ceil[\Big]{ \log_2 (c_d^2\sigma^{-d/2}\epsilon^{-1}) }
  \end{equation*}
and from $\frac 1{\ln 2}\leq 2$, with $C_d\defeq 2+4\ln c_d$.
\end{proof}

We can now prove Theorem~\ref{t:rd}.

\begin{proof}[Proof of Theorem~\ref{t:rd}]
  Notice that 
  \begin{equation*}
    \lambda
      \defeq \sup_{y \in (0, 1]}  \frac{f(y)}{y}
  \end{equation*}
  is finite, and 
  if $\varphi$ solves~\eqref{e:ad}
on $\mathbb T^d$  with 
initial data $\phi_0\defeq \theta_0$, then
  the comparison principle shows that $\theta_t \leq e^{\lambda t} \varphi_t$ for all $t\geq 0$.
  Let $t_0\defeq -\frac 1\lambda \ln \frac{\alpha_0+\bar\theta_0}{2\alpha_0}>0$ and
  \begin{equation*}
    \epsilon \defeq \alpha_0 e^{-\lambda t_0} - \bar \theta_0 =  \frac{\alpha_0  - \bar \theta_0}2>0\,.
  \end{equation*}
  If now $\tau_0>0$ is such that for any $\tau_*(u)\leq \tau_0$ we have that the right-hand side of \eqref{6.3} is at most $t_0$, then
Proposition~\ref{p:L1Linf} shows that
  \begin{equation*}
    \norm{\varphi_{t_0} - \bar \theta_0}_{L^\infty}
      \leq \epsilon  \norm{\theta_0 - \bar \theta_0}_{L^1}
      \leq \epsilon \leq \alpha_0 e^{-\lambda t_0} - \bar \theta_0\,.
  \end{equation*}
  Therefore $\norm{\theta_{t_0}}_{L^\infty} \leq\norm{\varphi_{t_0}}_{L^\infty} e^{\lambda t_0} \leq \alpha_0 $,  completing the proof of the first claim.  The last claim follows from Theorem \ref{C.7.6}.
\end{proof}

\section{Dissipation times of periodic and cellular flows}\label{s:dtimeBounds}

In this section we will prove Theorem \ref{C.7.6} and also provide examples of cellular flows in 2D and 3D satisfying its hypotheses.

Consider now the advection-diffusion equation
\eqref{e:ad}
on $\R^d$, with initial data $\tht_0$ and a time-independent mean-zero divergence-free Lipschitz vector field $u$.
Consider also the stochastic process $X_t^{x}=X_t^x(\omega)$ satisfying the SDE
\begin{equation} \label{7.2}
dX^{x}_t=\sqrt{2}\,dB_t - u(X^{x}_t)dt, \qquad X^{x}_0=x.
\end{equation}
Here $B_t=B_t(\omega)$ is a normalized Brownian motion on $\R^d$ with $B_0=0$, defined on some probability space $(\Omega,\calB_\infty,\bbP_\Omega)$.  Lemma 7.8 in \cite{Oksendal03} shows that if $k_t(x,y)$ is the fundamental solution for \eqref{e:ad} (i.e., $k_t(x,\cdot)$ is the density for $X_t^x$) and $\bbE_\Omega$ the expectation with respect to $\omega\in\Omega$, then solutions to \eqref{e:ad} satisfy
\begin{equation} \label{7.3}
\tht(t,x) = \int_{\R^d} k_t(x,y) \tht_0(y) \,dy = \bbE_{\Omega} \big( \tht_0(X^{x}_t) \big).
\end{equation}
Finally, for each vector $e\in\R^d$, the {\it effective diffusivity} of $u$ in direction $e$ is the number
\begin{equation} \label{7.4}
\diffe(u) \defeq \lim_{t\to\infty} \bbE_\Om \left(\frac {\left|(X_t^x-x)\cdot e\right|^2}{2t} \right) \qquad (\geq 1),
\end{equation}
with the limit being independent of $x\in\R^d$.  It will be also convenient to denote  
\[
D(u) \defeq \min\{D_{e_1}(u),\dots,D_{e_d}(u)\}\geq 1
\]
 the minimum of effective diffusivities in all the coordinate directions. We refer the reader to the discussion in Sections 1 and 2 of \cite{Zlatos11} for more details.  
 
It follows from \eqref{7.4} that the stochastic process travels far (relative to $\sqrt t$, for large $t$) with large probability when the effective diffusivity is large, which may of course aid mixing when such flows are scaled down and acting on a torus.  In fact, cellular flows in two dimensions have their effective diffusivities growing proportionally to the square root of their amplitudes \cite{FannjiangPapanicolaou94,Koralov04}, so this suggests that large amplitude cellular flows (with small cells) should be good short time mixers on $\T^2$.  To show that, we need to use the following result from \cite{Zlatos11}, which is a short time large probability one-sided analogue of \eqref{7.4} for 1-periodic flows, with $u$-independent constants.

\begin{lemma}[Theorem 2.1 in \cite{Zlatos11}] \label{L.7.1}
There is $C\geq 1$ such that  for any 1-periodic incompressible mean-zero Lipschitz flow $u$ on $\R^d$, any $e\in\R^d$, any  $\alpha>0$, and any $\tau\geq 1$,  there are $t\in [1,\tau+1]$ and $x\in\R^d$ such that 
\begin{equation} \label{7.5}
\bbP_\Om \left(\left|(X_{t}^x-x)\cdot e\right| \geq\alpha \sqrt{\tau\diffe(u)} \right) \geq 1-C\alpha.
\end{equation}
\end{lemma}

\begin{remark*}
 The result in \cite{Zlatos11} has $t\in [0,\tau]$, but the proof can be easily modified to obtain this version.
\end{remark*}

Because we want to consider general periods, we now extend this result to that case.

\begin{lemma} \label{L.7.2}
There is $C\geq 1$ such that  for any $l$-periodic incompressible mean-zero Lipschitz flow $u$, any $e\in\R^d$, any  $\alpha>0$, and any $\tau\geq l^2$, there are $t\in [l^2,\tau+l^2]$ and $x\in\R^d$ such that \eqref{7.5} holds.
\end{lemma}

\begin{proof}
Let $v(x) \defeq l\, u(l x)$, let $X_t^x$ be from \eqref{7.2}, and let $Y_t^x \defeq \frac 1l X^{lx}_{l^2t}$. Then $Y_t^x$ satisfies 
\[
dY^{x}_t=\sqrt{2}\,\frac 1ldB_{l^2t} - v(Y^{x}_t)dt, \qquad Y^{x}_0=x.
\]
Since $\frac 1lB_{l^2t}$ equals $B_t$ in law, it follows that $Y_t^x$ is a stochastic process corresponding to the 1-periodic flow $v$ via \eqref{7.2}.  
From \eqref{7.4} we immediately see that $D_e(v)=D_e(u)$ for all $e\in\mathbb R^d$, and Lemma \ref{L.7.1} applied to $v$ then finishes the proof.
\end{proof}

Next we extend the claim to all $x$, in an appropriate sense.  Let $\Psi(x) \defeq \int_x^{\infty} \frac 1{\sqrt{2\pi}}e^{-y^2/2}dy$.

\begin{theorem} \label{T.7.3}
There is $C\geq 1$ such that  for any $l$-periodic incompressible mean-zero Lipschitz flow $u$ on $\R^d$, any $e\in\bbS^d$, any  $\alpha\in(0,1)$, and any $\tau\geq l^2$, there is $t\in [l^2,\tau+l^2]$ such that for any $x\in\R^d$ we have
\[
\bbP_\Om \left(\left|(X_{t}^x-x)\cdot e\right| \geq \alpha \sqrt{\tau\diffe(u)} - 3 l \Psi^{-1}(\alpha)  - 2l^2\|u\|_{L^\infty} \right) \geq 1-C\sqrt{\alpha}.
\]
  Also,  $D_e(u)=D_e(u^L)$ for any $L>0$, where $u^L(x) \defeq \frac 1Lu(\frac xL)$.
\end{theorem}

\begin{remark*}
Note that for any 1-periodic flow $u$ we now have 
\[
\lim_{L\to 0} \left( 3 L \Psi^{-1}(C\alpha)  - 2L^2\|u^L\|_{L^\infty} \right)= 0,
\]
with $u^L$ being $L$-periodic.  So if $u$ has a large effective diffusivity, $X_t^x$ will travel far for  all $t$ as small as one needs and all $x\in\R^d$, provided we scale the flow down sufficiently (and multiply by the same scaling factor).  This will result in good mixing properties of such flows on $\T^d$.
\end{remark*}

\begin{proof}
The last claim was established in the previous proof.

It is well known that there is $c>0$ such that for any $l$-periodic flow $u$, the probability density function $h_t(x,y) \defeq \sum_{k\in(l\mathbb Z)^d} k_t(x,y+k)$  of the process $X_{t}^x\,{\rm mod}\, l \in (l\T)^d$ takes values in $[cl^{-d},c^{-1}l^{-d}]$ when $t=l^2$ (for each $x\in\T^d$).  
Given $\tau\geq l^2$, take $(t,x)$ from Lemma \ref{L.7.2}.  Then for any $\alpha>0$ and $C$ from Lemma \ref{L.7.2} we have
\begin{align*}
& 1- C\alpha -  \bbP_\Omega \left(|\sqrt 2 B_1\cdot e|> a \right) 
\\  &\leq \int_{(l\T)^d} h_{l^2}(x,y) \bbP_\Om \left(\left|(X_{t-l^2}^y-y)\cdot e\right| \geq \alpha \sqrt{\tau\diffe(u)} - l a - l^2\|u\|_{L^\infty} \right) dy.
\end{align*}
This is due to the Markov property of $X_t^x$ as well as the fact that if $|\sqrt 2B_{l^2}\cdot e|\leq la$ (which has the same probability as $|\sqrt 2B_1\cdot e|\leq a$), then $|X_{l^2}^x-x|\leq la + l^2 \|u\|_{L^\infty}$. 
Let us now pick $a \defeq \sqrt 2\Psi^{-1}(\alpha)$, so we have $\bbP_\Omega(|\sqrt 2B_1\cdot e|> a)\leq 2C\alpha$. Using $h_{l^2}\geq cl^{-d}$ and $\int_{(l\T)^d} h_{l^2}(x,y)dy=1$, we find that the measure of the set of all $y\in(l\T)^d$ such that
\[
\bbP_\Om \left(\left|(X_{t-l^2}^y-y)\cdot e\right| \geq \alpha \sqrt{\tau\diffe(u)} - l a - l^2\|u\|_{L^\infty} \right)< 1-3c^{-1}\sqrt{C\alpha} 
\]
is at most $\sqrt{C\alpha} \,l^d$.

Let now $z\in\R^d$ be arbitrary.  Then Markov property again yields 
\begin{align*}
  \MoveEqLeft
\bbP_\Om \left(\left|(X_{t}^z-z)\cdot e\right| \geq \alpha \sqrt{\tau\diffe(u)} - 2l a - 2l^2\|u\|_{L^\infty} \right)
  \\
  &\qquad + \bbP_\Omega \left(|\sqrt 2 B_1\cdot e|> a \right) 
  \\
  & \geq \int_{(l\T)^d} h_{l^2}(z,y) \cdot
  \\
    &\qquad\qquad\cdot \bbP_\Om \left(\left|(X_{t-l^2}^y-y)\cdot e\right| \geq \alpha \sqrt{\tau\diffe(u)} - l a - l^2\|u\|_{L^\infty} \right) \, dy
  \\
  & \geq 1-3c^{-1}\sqrt{C\alpha}  - c^{-1}l^{-d}\sqrt{C\alpha} \,l^d \,.
\end{align*}
So $C\alpha\leq 1\leq c^{-1}$ shows for the above $t$ and all $z\in\R^d$,
\[
\bbP_\Om \left(\left|(X_{t}^z-z)\cdot e\right| \geq \alpha \sqrt{\tau\diffe(u)} - 2l a - 2l^2\|u\|_{L^\infty} \right) \geq 1-5c^{-1}\sqrt{C\alpha}.
\]
If we now change $C$ to be $5\max\{C,c^{-2}\}$, the result follows.
\end{proof}

From now on we will consider flows on $\T^d$ with period $\frac 1{\nu}$ ($\nu\in\mathbb N$) that are also symmetric.  We say that a flow $u$ is symmetric in $x_n$ if we have
\begin{equation} \label{7.6}
u(R_nx)=R_nu(x) \qquad \text{for all  $x\in\T^d$},
\end{equation}
where $R_nv \defeq v-2v_ne_n$ for $v=\sum_{n=1}^d v_ne_n\in\R^d$ is the reflection in the $n^{\rm th}$ coordinate.  
Note that a periodic flow that is symmetric in all $d$ coordinates has a cellular structure.


We let $X_t^x$ be the process above, when $u$ is considered on $\R^d$ (extended periodically),  and note that $X_{t}^x\,{\rm mod}\, 1$ is the corresponding process on $\T^d$.
  Finally, for any divisor $\mu$ of $\nu$, we denote by $\calC^\mu_{k} \defeq \Pi_{n=1}^d[\frac {k_n}{\mu},\frac {k_{n}+1}{\mu})$ ($k=(k_1,\dots,k_d)\in\{0,\dots,\mu-1\}^d$) the ``cells'' of $u$ on $\T^d$ of size $\frac 1\mu$ (each of which is an invariant set for the flow when $u$ is symmetric).

\begin{lemma} \label{L.7.4}
There is $C\geq 1$ such that  for any $\frac 1{\nu}$-periodic incompressible mean-zero Lipschitz flow $u$ that is symmetric in $x_n$  we have
\begin{equation}\label{7.7}
\left| \bbP_\Om \left(X_\tau^x\in \calC^\mu_{k}\right) - \bbP_\Om \left(X_\tau^x\in \calC^\mu_{m}\right) \right| \leq C\sqrt{\alpha}
\end{equation}
for any $x\in\R^d$, any $\alpha\in(0,1)$, any divisor $\mu$ of $\nu$, and any $k,m$ as above such that $k-m$ is a multiple of $e_n$, provided
\begin{equation}\label{7.8}
\tau \geq   \frac{(6 \nu \Psi^{-1}(\alpha)  + 4\|u\|_{L^\infty} + \nu^2)^2}{4\nu^4 \alpha^2 D_{e_n}(u)}  + \frac 2{\nu^2}.
\end{equation}
In particular, if $u$ is symmetric in all coordinates, then we have 
\begin{equation}\label{7.9}
\left| \bbP_\Om \left(X_\tau^x\in \calC^\mu_{k}\right) - \frac 1{\mu^d} \right| \leq Cd\sqrt{\alpha}
\end{equation}
for any $x\in\R^d$, any $\alpha,\mu$ as above, and any $k,m\in\{0,\dots,\mu-1\}^d$, provided
\begin{equation}\label{7.10}
\tau \geq   \frac{(6 \nu \Psi^{-1}(\alpha)  + 4 \|u\|_{L^\infty} + \nu^2)^2}{4\nu^4 \alpha^2 D(u)}  + \frac 2{\nu^2}.
\end{equation}
\end{lemma}

\begin{proof}
The second claim follows from the first applied successively in all $d$ coordinates, and from $\sum_m\bbP_\Om \left(X_\tau^x\in \calC^\mu_{m}\right)=1$.  

As for the first claim, notice that \eqref{7.7} is implies $\tau\geq \frac 2{\nu^2}$ and
\[
\alpha \sqrt{(\tau -\nu^{-2} )D_{e_j}(u)}  - 3 \nu^{-1} \Psi^{-1}(\alpha)  - 2\nu^{-2}\|u\|_{L^\infty} \geq \frac 12.  
\]
Theorem \ref{T.7.3} shows that there is $s\in [\frac 1{\nu^2},\tau]$ such that 
\[
\bbP_\Om \left(\left|(X_{s}^x-x)_n\right| \geq \frac 12 \right) \geq 1-C\sqrt{\alpha},
\]
when $X_t^x$ is considered on $\R^d$.
This means that there is $\Om'\subseteq\Om$ with $\bbP_\Om(\Om')\geq 1-C\sqrt{\alpha}$ such that for each $\omega\in\Om'$, the process $X_{t}^x(\omega)\,{\rm mod}\, 1$ hits at least one of the two hyperplanes on $\T^d$ with $x_n \in\{ \frac{k_n+m_n+1}{2\mu}, \frac{k_n+m_n+1}{2\mu}+\frac 12\}$ by time $s$. 

 But symmetry in $x_n$ and $\frac1\nu$-periodicity show that $u$ is symmetric across any hyperplane with $x_n=\frac j{2\nu}$ ($j=0,\dots,2\nu-1$), that is,
\[
u \left(R_n\left(x-\frac j{2\nu}e_n\right)+\frac j{2\nu}e_n\right) = u \left(R_nx+\frac j{\nu}e_n\right) = u(R_nx) = R_nu(x)
\]
for all $x\in\T^d$.  Since also $\calC^\mu_k$ and $\calC^\mu_m$ are mapped into each other by reflections across the hyperplanes with $x_n \in\{ \frac{k_n+m_n+1}{2\mu}, \frac{k_n+m_n+1}{2\mu}+\frac 12\}$, the law of $B_t$ is invariant under the reflection $R_n$, and $X_t^x$  satisfies the strong Markov property, it follows that
\[
\bbP_\Om(\omega\in\Om'  \text{ and } X_\tau^x(\omega)\in \calC^\mu_k) = \bbP_\Om(\omega\in\Om'  \text{ and } X_\tau^x(\omega)\in \calC^\mu_m)
\]
because $\tau\geq s$.  This finishes the proof.
\end{proof}

If $h_t$ is as above and $h'_t$ is the corresponding probability density kernel for the flow $-u$ on $\T^d$, then we have $h_t(x,y)=h'_t(y,x)$ for all $x,y\in\T^d$ because $\nabla\cdot u\equiv 0$ means that the adjoint of the operator $e^{t(\lap-u\cdot\nabla)}$ (whose kernel is $h_t$) is $e^{t(\lap+u\cdot\nabla)}$ (whose kernel is $h'_t$).  Therefore, the solution to \eqref{e:ad} on $\T^d$  satisfies
\begin{align*}
\int_{\calC^\mu_{k}} \tht(t,x)dx & = \int_{\calC^\mu_{k}}\int_{\T^d} h_t(x,y)\tht_0(y) dydx 
\\ &= \int_{\T^d} \int_{\calC^\mu_{k}} h'_t(y,x) dx \, \tht_0(y) dy 
\\ &=  \int_{\T^d} \bbP_\Om \left(X_\tau^y\in \calC^\mu_{k}\right) \, \tht_0(y) dy,
\end{align*}
with $X_t^y$ the process corresponding to $-u$.
If \eqref{7.9} holds and $\tht_0$ is mean-zero, then we obtain
\[
\left| \int_{\calC^\mu_{k}} \tht(\tau,x)dx \right| \leq Cd\sqrt\alpha \|\tht_0\|_{L^1}
\]
for each $k$.  Poincar\'e inequality now shows that
\[
 \int_{\calC^\mu_{k}} |\nabla \tht(\tau,x)|^2dx  
 \geq C_d^{-2}\mu^2\inf_{|a|\leq C_d\sqrt\alpha \mu^d \|\tht_0\|_{L^1}}  \int_{\calC^\mu_{k}} |\tht(\tau,x)-a|^2dx 
\]
for each $k$ and some $C_d\geq 1$ (only dependent on $d$).  But then
\begin{align*}
\mu^{-1} \| \nabla \tht_\tau \|_{L^2} 
&\geq C_d^{-1}\left\| \left( |\tht_\tau|-C_d\sqrt\alpha \mu^d \|\tht_0\|_{L^1} \right)^+ \right\|_{L^2} 
\\ &\geq C_d^{-1} \|\tht_\tau\|_{L^2} - \sqrt\alpha\mu^d \|\tht_0\|_{L^1},
\end{align*}
which now implies that
\begin{equation}\label{7.11}
\frac d{d\tau} \|\tht_\tau\|_{L^2}^2 = -2\|\nabla\tht_\tau\|^2_{L^2} \leq -  C_d^{-1} \mu \|\tht_\tau\|_{L^2}
\end{equation}
as long as $ \|\tht_\tau\|_{L^2} \geq 2C_d \sqrt\alpha\mu^d \|\tht_0\|_{L^2}$ (because $\|\tht_0\|_{L^2}\geq \|\tht_0\|_{L^1}$).  From this and Lemma \ref{L.7.4} we obtain the following result.

\begin{theorem} \label{T.7.5}
For any $d\in\mathbb N$, there is $C_d\geq 1$ such that  for any $\frac 1{\nu}$-periodic incompressible mean-zero Lipschitz flow $u$ on $\T^d$ that is symmetric in all coordinates and any mean-zero $\tht_0\in L^2(\T^d)$, the solution to \eqref{e:ad} on $\T^d$ satisfies
\begin{equation}\label{7.12}
\|\tht_t\|_{L^2(\T^d)} \leq 2C_d \sqrt\alpha\mu^d \|\tht_0\|_{L^2(\T^d)}
\end{equation}
 for each $\alpha\in(0,1)$ and each divisor $\mu$ of $\nu$, provided
\begin{equation}\label{7.13}
t \geq   \frac{(6 \nu \Psi^{-1}(\alpha)  + 4 \|u\|_{L^\infty} + \nu^2)^2}{4\nu^4 \alpha^2 D(-u)}  + \frac 2{\nu^2} + \frac{2C_d}\mu \ln^+ \frac {1} {2 C_d\sqrt\alpha\mu^d}
\end{equation}
\end{theorem}

\begin{proof}
This follows from \eqref{7.11} for all times $\tau$ satisfying \eqref{7.10} such that $ \|\tht_\tau\|_{L^2} \geq 2C_d \sqrt\alpha\mu^d \|\tht_0\|_{L^1}$, noting also that $ \|\tht_\tau\|_{L^2}\leq \|\tht_0\|_{L^2}$.
\end{proof}

We can now prove Theorem~\ref{C.7.6}.

\begin{proof}
Let $\alpha_n \defeq D(u_n)^{-1/4}$ and $\mu_n \defeq \lfloor \alpha_n^{-1/4d} \rfloor$.  Using these values for $\alpha,\mu$ in Theorem \ref{T.7.5}, as well as $u(x)=-\nu_n u_n(\nu_n x)$ with a sufficiently large $\nu=\nu_n$ (a multiple of $\mu_n$), together with the last claim in Theorem \ref{T.7.3}, yields that the sum of the first two fractions in \eqref{7.13} is less than $D(u_n)^{-1/2}$, while $2C_d \sqrt\alpha\mu^d$ in \eqref{7.12} is bounded above by $2C_d D(u_n)^{-1/16}$ and the last term in \eqref{7.13} is bounded above by $\frac C{\mu_n} \ln 2\mu_n$ for some $C\geq 1$ and all  $n$.  Since all three bounds converge to 0 as $n\to\infty$, the proof of the first claim is finished.

 The last claim follows from this and Examples \ref{E.7.7} and \ref{E.7.8} below.
\end{proof}

\begin{remark}\label{r:tauDSize}
  We can in fact pick $\nu_n$ to be the smallest multiple of $\mu_n=\lfloor D(u_n)^{1/16d} \rfloor$ greater than $\norm{u_n}_{L^\infty} D(u_n)^{(1-8d)/32d}$ and still have the right-hand side of \eqref{7.13} bounded above by $\frac C{\mu_n} \ln 2\mu_n$ for some $C\geq 1$.  That is, we have $\tau_*(v_n)\leq \frac C{\mu_n} \ln 2\mu_n$.  In particular, if $u_n$ are the 2D flows from Example \ref{E.7.7} below, we obtain $\nu_n\sim A_n^{113/128}$ and $\tau_*(v_n)\leq CA_n^{-1/64} \ln A_n$
\end{remark}

\begin{example} \label{E.7.7}
  When $d=2$, hypotheses of Theorem \ref{C.7.6} are satisfied  by $u_n(x,y) \defeq A_n\nabla^\perp \psi(x,y)$ with $\lim_{n\to\infty} A_n=\infty$ and  any stream function $\psi\in C^{2,\delta}(\T^d)$ that has only non-degenerate critical points and vanishes when $x=0$ or $y=0$, for instance $\psi(x,y) = \sin(2\pi x)\sin(2\pi y)$. (When periodically extended onto $\R^2$, these flows are called cellular, with their cells being $(k,k+1)\times(m,m+1)$ for any $k,m\in\mathbb Z$.)  This is because then $D(u_n)\sim A_n^{1/2}$ by~\cite{FannjiangPapanicolaou94,Koralov04}.
\end{example}
\begin{example} \label{E.7.8}
  When $d=3$, hypotheses of Theorem \ref{C.7.6} are satisfied by 
\[
u_n(x,y,z)  \defeq  A_n \big( \Phi_{x}(x,y) W'(z),
\Phi_{y}(x,y) W'(z), 8\pi^2\Phi(x,y)W(z) \big)
\]
with $\lim_{n\to\infty} A_n=\infty$, and the functions  $\Phi(x,y)  \defeq  \cos (2\pi x)\cos (2\pi y)$ and $W(z)  \defeq  \sin (2\pi z)$ (this again extends periodically to a cellular flow on $\R^3$).  We now have $\lim_{n\to\infty} D(u_n)=\infty$ by~\cite{RyzhikZlatos07}.
\end{example}

\bibliographystyle{halpha-abbrv}
\bibliography{refs,preprints}
\end{document}